\documentclass{amsart}
\usepackage{amsfonts}
\usepackage{amssymb}
\usepackage{amsmath}

\newtheorem{theorem}{Theorem}

\newtheorem{corollary}[theorem]{Corollary}

\newtheorem{example}[theorem]{Example}

\newtheorem{lemma}[theorem]{Lemma}

\newtheorem{proposition}[theorem]{Proposition}
\newtheorem{remark}[theorem]{Remark}

\begin{document}
\title{Regular elements determined by generalized inverses}
\author{Adel Alahmadi, S. K. Jain, and Andr\'e Leroy}
\address{Adel Alahmedi\\
Department of Mathematics\\
King abdulaziz University\\
Jeddah, SA\\
Email:adelnife2@yahoo.com;\\
S. K. Jain\\
Department of Mathematics\\
King Andulaziz University Jeddah, SA,and\\
Ohio University, USA\\
Email:jain@ohio.edu\\
Andre Leroy\\
Facult\'e Jean Perrin\\
Universit\'e d'Artois\\
Lens, France\\
Email:andre.leroy@univ-artois.fr}

\begin{abstract}
In a semiprime ring, von Neumann regular elements are determined by their
inner inverses. In particular, for elements $a,b$ of a von Neumann regular ring $R$, $a=b$ if
and only if $I(a)=I(b)$, where $I(x)$ denotes the set of inner
inverses of $x\in R$. We also prove that, in a semiprime ring, the same is
true for reflexive inverses.
\end{abstract}

\maketitle



\section{Introduction and preliminaries}

\bigskip

In this short note, $R$ will stand for an associative ring with unity. An
element $a\in R$ is (von Neumann) regular if there exists $x\in R$ such that 
$a=axa$. Such an element $x$ is called an inner inverse (also called von
Neumann inverse or generalized inverse) of $a$. The set of regular elements
of a ring $R$ is denoted by $Reg(R)$. A ring $R$ is regular if $Reg(R)=R$.
Note that a regular ring is semiprime. In general, a regular element may
have more than one inner inverse. We denote the set of inner inverses of $a$
by $I(a)$. An element $x\in R$ is called an outer inverse of $a$ if $xax=x$.
Note that if $x\in I(a)$ then $xax$ is both an inner and an outer inverse of 
$a$. An element $x\in R$ is called a reflexive inverse of $a$ if it is both
an inner and an outer inverse of $a$. Denote the set of reflexive inverses
of $a$ by $Ref(a)$. We first obtain a necessary and sufficient condition for 
$I(a)\subseteq I(b)$ (Lemma \ref{Comparison of quasi inverse set}) and use
this to prove that in a semiprime ring, for $a,b\in Reg(R)$, $I(a)=I(b)$ if
and only if $a=b$ if and only if $Ref(a)=Ref(b)$. (Theorem \ref{I(a)=I(b)}
and Theorem \ref{ref(a)=ref(b)} ).

We begin with a few key lemmas.

The following is well-known (cf. \cite{BIG} Corollary 1, Chapter 2. p. 40.)

\begin{lemma}
For $a\in R$ and $a_{0}\in I(a),$ we have $I(a)=\{a_{0}+t-a_{0}ataa_{0}\mid
t\in R\}.$
\end{lemma}

As usual, $l(a)$ and $r(a)$ denote respectively the left and right
annihilator of an element $a\in R$. We define the inner annihilator of an
element $a\in R$, as $\{x\in R\mid axa=0\}$ and denote it by $Iann(a)$.

The next Proposition gives a link between $I(a)$ and $Ref(a)$.

\begin{proposition}
\label{I(a) and ref(a)} For $a\in Reg(R)$, let $\varphi_a:I(a)
\longrightarrow Ref(a)$ be such that $\varphi_a(x)=xax$. Then

\begin{enumerate}
\item The map $\varphi_a$ is onto.

\item $Ref(a)=I(a)aI(a)$.

\item If $x,y\in I(a)$ are such that $\varphi_a(x)=\varphi_a(y)$ then $%
x-y\in l(a)\cap r(a)$.

\item Let $x\in Ref(a)$, then $\varphi_a(x)=x$.
\end{enumerate}
\end{proposition}

\begin{proof}
This is clear.
\end{proof}

The next lemma is straightforward.

\begin{lemma}
Let $a\in Reg(R)$ and $a_{0}\in I(a).$ \ Write $e=aa_{0},$ $f=a_{0}a$ and $%
e^{\prime }=1-e,$ $f^{\prime }=1-f$. Then

\begin{enumerate}
\item[(i)] $Iann(a)=l(a)+r(a)=Re^{\prime}+f^{\prime}R$.

\item[(ii)] $I(a)= a_0+Iann(a)=a_0 + Re^{\prime}+f^{\prime}R$.

\item[(iii)] If $a_{0}\in Ref(a),$ then $Ref(a)=a_{0}+fRe^{\prime
}+f^{\prime }Re+f^{\prime }RaRe^{\prime }$.
\end{enumerate}
\end{lemma}

\begin{proof}
We simply mention that the statement (iii) can be proved by using statement
(2) of the above proposition \ref{I(a) and ref(a)}.
\end{proof}

\section{Characterization of $I(a)$ and $Ref(a)$}

We can now state our first result in the following proposition.

\begin{proposition}
\label{Invariance} Let $R$ be a semiprime ring. If $a\in Reg(R)$, then for
any $b\in R,$ $bI(a)b$ is a singleton set if and only if $b\in Ra\cap aR$.
\end{proposition}

\begin{proof}
Firstly, suppose that there exist $x,y\in R$ such that $b=xa=ay$ and let $%
a_{0}\in I(a)$. We then have that, for any $t\in R$, $%
b(a_{0}+t-a_{0}ataa_{0})b=(xaa_{0}+xat-xataa_{0})ay=xay+xatay-xatay=xay$.
This shows that indeed $bI(a)b$ is a singleton set.

\noindent Conversely, suppose that $bI(a)b=\{ba_{0}b\}$. We then have $%
b(a_{0}+t-a_{0}ataa_{0})b=ba_{0}b$, for any $t\in R$. This implies that, for
any $t\in R$, we have 
\begin{equation*}
b(t-a_{0}ataa_{0})b=0.
\end{equation*}%
Substituting $(1-a_{0}a)t$ for $t$ in this equality leads to $%
b(1-a_{0}a)tb=0 $, for any $t\in R$. The semiprimeness of $R$ then implies
that $b(1-a_{0}a)=0$, i.e. $b=ba_{0}a\in Ra$. Similarly, substituting $t$ by 
$t(1-aa_{0})$ in the above equality gives $b=aa_{0}b$. In particular, $b\in
aR$.
\end{proof}

\vspace{3mm}

We recall the following result obtained by S.K. Jain and M. Prasad (\cite{JP}%
).

\begin{lemma}
\label{Jain-Prasad} Let $R$ be a ring and let $b,d\in R$ such that $b+d$ is
a Von Neumann regular element. Then the following are equivalent:

\begin{enumerate}
\item $bR\oplus dR=(b+d)R$.

\item $Rb\oplus Rd = R(b+d)$.

\item $bR \cap dR = \{0\}$ and $Rb\cap Rd=\{0\}$.
\end{enumerate}
\end{lemma}

The next proposition provides necessary and sufficient conditions as to when 
$I(a)\subseteq I(b),$ where $a,b\in Reg(R)$ and $R$ is semiprime.

\begin{proposition}
\label{Comparison of quasi inverse set} Let $R$ be a semiprime ring and let $%
a,b\in Reg(R)$. Then $I(a)\subseteq I(b)$ if and only if $bR\cap dR=0$ and $%
Rb\cap Rd=0$ where $a=d+b$.
\end{proposition}

\begin{proof}
Since $I(a)\subseteq I(b)$, we have $bxb=b$ for every $x\in I(a)$. By
Proposition \ref{Invariance}, $b\in Ra\cap aR$. Write $b=\alpha a=a\beta $
for some $\alpha,\beta \in R$. Then $bI(a)a=b.$ Next, $%
bI(a)d=bI(a)a-bI(a)b=b-bI(a)b=0.$ 
Consider now $dI(a)b=aI(a)b-bI(a)b=a\beta -bI(a)b=b-b=0.$ We thus have 
\begin{equation*}
bI(a)d=0 \quad \mathrm{and} \quad dI(a)b=0 \quad \quad \quad (1)
\end{equation*}
Then, for any $x\in I(a)$, we have $b+d=a=axa=(b+d)x(b+d)=bxa+dxb+
dxd=b+0+dxd.$ This yields, 
\begin{equation*}
dI(a)d=d \quad \quad \quad \quad \quad \quad \quad \quad \quad (2)
\end{equation*}
Now, we proceed to show $dR\cap bR=0.$ Let $bx=dy\in bR\cap dR.$ Multiplying
both sides of the equality (2) by $y$ on the right and using $bx=dy$ we
obtain $dI(a)bx=dy.$ As proved above, we have $dI(a)b=0.$ and so $dy=0$.
This proves our assertion.

\noindent Similarly, we show that $Rb\cap Rd=0$. Let $xb=yd\in Rb\cap Rd.$
Now, multiplying both sides of the equality (2) on the left by $y$, we get $%
ydI(a)d=yd$. This proves $xbI(a)d=yd.$ Since $bI(a)d=0$, we obtain $yd=0,$
proving $Rb\cap Rd=0$.

\noindent The converse is easy using the above lemma \ref{Jain-Prasad}.
\end{proof}

Next, we show, in particular, that the regular elements of a semiprime ring
are equal if their sets of inner inverses are the same.

\begin{theorem}
\label{I(a)=I(b)} Let $R$ be a semiprime ring and $a,b\in Reg(R)$. Then $%
I(a)=I(b)$ if and only if $a=b$.
\end{theorem}

\begin{proof}
We only need to prove the sufficency. So assume that $I(a)=I(b)$.
Proposotion \ref{Comparison of quasi inverse set} implies that we can write $%
a=b+d$ with $bR\cap dR=0,\; Rd\cap Rb=0$. 
Lemma \ref{Jain-Prasad} then gives that $(b+d)R=bR\oplus dR$. Since $%
I(a)=I(b)$ we also have $aI(b)a=\{a\}$ and $bI(a)b=\{b\}$ and Proposition %
\ref{Invariance} implies that $Ra=Rb$ and $aR=bR$. This leads to $%
aR=(b+d)R=bR\oplus dR=aR\oplus dR$. This forces $d$ to be zero and hence $%
a=b $, as desired.

Alternatively we may invoke Hartwig's result (cf. \cite{H}) in place of
Lemma \ref{Jain-Prasad}. 
This was pointed out to us by T.Y. Lam. Indeed, by our Proposition \ref%
{Invariance} we have $aR=bR$ and $Ra=Rb$, and thus by invoking Hartwig's
result, there exist units $u,v\in R$ such that $b=au=va$. If $x\in I(a)=I(b)$%
, then $axa=a$ and $bxb=b$. The last equality implies that $vaxau=au$ and
hence $va=a$. Thus $b=a$.
\end{proof}

\vspace{3mm}

\begin{corollary}
Let $R$ be a regular ring. Then $I(a)=I(b)$ if and only if $a=b$.
\end{corollary}

\vspace{2mm}


\vspace{2mm}

\begin{remark}
{\rm Pace Nielsen remarked that, in the above theorem, the semiprime
hypothesis can be replaced by the assumption that $a-b$ is regular. So
assume $I(a)=I(b)$, and $a-b\in Reg(R)$. As in our Proposition \ref%
{Invariance} we obtain 
\begin{equation*}
bt(1-aa_0)b=0,\quad\quad\quad (1)
\end{equation*}
for any $t\in R $. If $b_1$ is a reflexive inverse of $b$, we obtain, for
any $t$ in $R$, $ataa_0b=(ab_1a)taa_0b=a(b_1bb_1)ataa_0b= ab_1(bb_1ataa_0b)$%
. Replace $t$ by $b_1at$ in (1) and obtain 
\begin{equation*}
ataa_0b=atb. \quad\quad\quad (2)
\end{equation*}
For $z\in I(a-b)$ we have $bzb=bza + azb-aza+a-b$. Using this equality we
compute 
\begin{equation*}
bzb=bzba_0b=(bza+azb-aza +a-b)a_0b=bzaa_0b+azb-azaa_0b+aa_0b-b.
\end{equation*}
Using formulae (1) and (2) we get $aa_0b=b$ so that $bR\subseteq aR$.
Symmetric arguments leads to $aR=bR$ and $Ra=Rb$ and Hartwig's theorem
finishes the proof. 
}
\end{remark}

\bigskip

\section{Reflexive inverses for semiprime rings}

We conclude by characterizing the equality of $Ref(a)=Ref(b),$ and obtain the analogue
of Theorem \ref{I(a)=I(b)} for reflexive inverses of semiprime rings.

\begin{theorem}
\label{ref(a)=ref(b)} Let $R$ be a semiprime ring such that $a,b\in Reg(R)$.
Then $Ref(a)=Ref(b)$ if and only if $a=b$.
\end{theorem}

\begin{proof}
Let $a_0\in Ref(a)=Ref(b)$. Since $a=0$ if and only if $Ref(a)=0$, we may assume
that $a$ and $b$ are not zero. Since $bRef(a)b=bRef(b)b=b$ and $%
Ref(a)=I(a)aI(a) $, we have that, for any $t$ in $R$, 
\begin{equation*}
b(a_0 + t - a_0ataa_0) a (a_0 + t - a_0ataa_0)b=b \quad \quad \quad (1)
\end{equation*}
Replacing $t$ by $(1-a_0a)t$ and noting that $a(1-a_0a)=0$, we obtain
successively $b(a_0a + (1-a_0a)ta)(a_0 +(1-a_0a)t)b=b$ and $b(a_0a +
(1-a_0a)ta)(a_0)b=b$ and so $ba_0b + b(1-a_0a)taa_0b=b$. Since $ba_0b=b$
this gives $b(1-a_0a)taa_0b=0$ for all $t\in R$. This leads to 
\begin{equation*}
aa_0b(1-a_0a)taa_0b(1-a_0a)=0 \quad \forall t\in R.
\end{equation*}%
The semiprimeness of $R$ implies that $aa_0b(1-a_0a)=0$. Left multiplying by 
$a_0\in ref(a)$, we get that $a_0b(1-a_0a)=0$ and hence since $a_0\in I(b)$
we conclude that $b(1-a_0a)=0$. 
Therefore we obtain that $Rb\subseteq Ra$ and by symmetry $Ra\subseteq Rb$
and hence $Ra=Rb$. In the same way replacing $t$ by $t(1-aa_0)$ in (1), we
obtain $aR=bR$. The Hartwig's Theorem then gives us that there exist
invertible elements $u,v\in R$ such that $a=bu$ and $b=av$. The argument at
the end of the proof of the semiprime case (cf. Theorem \ref{I(a)=I(b)})
proves the theorem.
\end{proof}

We now give an example of a ring,  showing that without the
semipriness hypothesis both of the above theorems are false.

\begin{example}
{\rm 
Consider the $\mathbb{F}_{2}$-algebra 
\begin{equation*}
R=\mathbb{F}_{2}\langle a,b,x\mid
axa=a,bxb=b,xax=x,xbx=x,a^{2}=b^{2}=ab=x^{2}=0\rangle 
\end{equation*}%
This ring is finite and $\{a,b,x,ax,bx,xa,xb,axb,bxa\}$ is a basis of $R$ as
an $\mathbb{F}_{2}$-vector space. It is easy to determine that $r(a)=r(b)\langle
a,b,ax,bx,axb,bxa\rangle $, $l(a)=l(b)=\langle a,b,xa,xb,axb,bxa\rangle $ $%
I(a)=x+R$, $I(b)=x+R$ , $ref(a)=\{x\}=ref(b)$. Of course, $(RaR)^{2}=0$%
, showing that $R$ is not semiprime. }
\end{example}

\smallskip
The next corollary is a direct consequence of 
Theorems \ref{I(a)=I(b)} and \ref{ref(a)=ref(b)}.
\begin{corollary}
Let $a,b$ be elements of a semiprime ring $R$.   Then the following are equivalent:
\begin{enumerate}
\item $I(a)=I(b)$,
\item $a=b$,
\item $ref(a)=Ref(b)$.
\end{enumerate}
\end{corollary}

\begin{remark}
{\rm 
We close with the following comment. The question of the equality of
two elements in a regular ring that have the same set of inner inverses
arose while the authors have been working on the question: if, for a regular
self-injective ring $R$, $I(c)=I(a)+I(b)$, $a,b,c\in R$, is it true that $c$
is unique? If not, obtain a complete solution for $\mathrm{c.}$ We will
discuss that in another paper which is in progress. 
}
\end{remark}

\bigskip

\centerline{ {\bf Acknowledgment} } \smallskip

We thank the referee for his/her valuable comments including the remark that
our proof of Theorem 10 for prime rings can be carried to semiprime rings as
well.

\end{document}